\documentclass{amsart}
\usepackage[utf8]{inputenc}
\usepackage{amssymb,bm,amsfonts, stmaryrd, amsmath, amsthm, enumerate, mathtools,shuffle,bbm}
\usepackage{hyperref}
\usepackage{xifthen}
\usepackage{diagbox}
\usepackage{caption,subcaption}
\usepackage{stackengine}
\stackMath
\usepackage[mathlines]{lineno}
\hypersetup{
 colorlinks=true,
 linkcolor=blue,
 filecolor=magenta, 
 urlcolor=black,
}
\usepackage{tikz-cd}

\newtheorem{theorem}[subsection]{Theorem}

\newtheorem{proposition}[subsection]{Proposition}
\newtheorem{corollary}[subsection]{Corollary}
\theoremstyle{definition}
\newtheorem{definition}[subsection]{Definition}
\theoremstyle{remark}
\newtheorem{remark}[subsection]{Remark}
\numberwithin{equation}{subsection}

\theoremstyle{plain}
\newtheorem{innercustomthm}{Theorem}
\newenvironment{customthm}[1]
  {\renewcommand{\theinnercustomthm}{#1}%
   \begin{innercustomthm}}
  {\end{innercustomthm}}

\newtheorem{innercustomprop}{Proposition}
\newenvironment{customprop}[1]
  {\renewcommand\theinnercustomprop{#1}%
   \begin{innercustomprop}}
  {\end{innercustomprop}}

\DeclareMathOperator{\depth}{depth}
\DeclareMathOperator{\edim}{edim}
\DeclareMathOperator{\projdim}{projdim}
\DeclareMathOperator{\Ass}{Ass}
\newcommand{\cat}[1]{{\mathsf{#1}}}
\DeclareMathOperator{\height}{height}
\DeclareMathOperator{\bigheight}{bigheight}
\DeclareMathOperator{\superheight}{superheight}
\newcommand{\Hom}[4][]{\operatorname{Hom}^{#1}_{#2}(#3,#4)}
\DeclareMathOperator{\level}{level}
\newcommand{\lotimes}{\otimes^{\mathbf{L}}}
\DeclareMathOperator{\f_rank}{f-rank}
\newcommand{\set}[2]{\left\{#1 \,\middle|\, #2\right\}}
\newcommand{\susp}{\Sigma}
\DeclareMathOperator{\Spec}{Spec}
\newcommand{\Tor}[4][]{\operatorname{Tor}_{#1}^{#2}(#3,#4)}
\DeclareMathOperator{\HH}{H}

\newcommand{\fm}{\mathfrak{m}}
\newcommand{\fn}{\mathfrak{n}}
\newcommand{\fp}{\mathfrak{p}}

\begin{document}

\title[Level Inequalities]{A lower bound on levels with \\ applications to Koszul complexes}

\author[A. Kekkou]{Antonia Kekkou}
\address{Department of Mathematics,
University of Utah, Salt Lake City, UT 84112, U.S.A.}
\email{kekkou@math.utah.edu}

\date{\today}

\begin{abstract}
In this paper, we establish a lower bound on the level of a perfect complex with $I$-power torsion homology on positive degrees and an $I$-power torsion minimal generator for $\HH_0(F)$. Examples are provided to demonstrate that the bound is optimal. This result is applied to improve existing lower bounds on the level of a Koszul complex on various classes of sequences.
\end{abstract}

\maketitle
\setcounter{tocdepth}{1}

\section{Introduction}
This paper concerns certain homological invariants of finite free complexes over commutative rings. We prove the following:
\begin{customthm}{\ref{theorem_level_ineq}} Let $R$ be a commutative noetherian local ring, $I$ a proper ideal in $R$ and 
\begin{equation*}
F \colon 0 \longrightarrow F_n \longrightarrow F_{n-1} \longrightarrow \cdots \longrightarrow F_1 \longrightarrow F_0 \longrightarrow 0
\end{equation*}
a finite free $R$-complex with $\HH_0(F) \neq 0$. If $\HH_i(F)$ is $I$-power torsion for $i\geq 1$ and a minimal generator of $\HH_0(F)$ is $I$-power torsion, then the following inequality holds:
\begin{equation*}
\level^R F  \geq \dim R-\dim R/I+1.
\end{equation*}
\end{customthm}

An element is said to be \emph{$I$-power torsion} if there exists an $s>0$ such that $I^s$ annihilates it. Accordingly, a module is said to be \emph{$I$-power torsion} if each of its elements is $I$-power torsion. The $R$-level of a finite free $R$-complex measures the minimal number of mapping cones required to construct the complex from finite free modules; see \cite[Section 2]{ABIM:2010}. The $R$-level, in a way, serves as a measure of the complexity of a complex. It is bounded above by the length of the complex; see \cite[Lemma 2.5.2]{ABIM:2010}. From this observation, we deduce that Theorem \ref{theorem_level_ineq} refines \cite[Theorem 2.2]{CF:2025} due to Christensen and Ferraro, which establishes that a complex $F$ as above has length at least $\dim R-\dim R/I$.

Another closely related result is due to Avramov, Iyengar, and Neeman in \cite[Theorem 4.2]{AIN:2018}, which uses different techniques. Their result is that the $R$-level of $F$ must be at least $\height I+1$. Theorem \ref{theorem_level_ineq} extends this result, as the inequality $\dim R-\dim R/I \geq \height I$ always holds. 

These results fall under the general framework of Evans and Griffith's version of the New Intersection Theorem \cite{EG:1981}, stated by Hochster in \cite{H:1983}. Evans and Griffith's version is a generalization of the New Intersection Theorem due to Peskine and Szpiro in \cite{PS:1974}, and Roberts in \cite{R:1987}. The proof of Theorem \ref{theorem_level_ineq} is derived by combining the proofs of two versions of the New Intersection Theorem, namely those in \cite{AGMSV:2017} by Altmann, Grifo, Montaño, Sanders, and Vue, and \cite{CF:2025} by Christensen and Ferraro.

This paper also establishes a lower bound on the $R$-level of certain differential graded $R$-algebras in terms of the free rank of the conormal module $I/I^2$, viewed as an $R/I$-module.

\begin{customprop}{\ref{level_frank_dgalgebra}}
Let $(R,\fm, k)$ be a commutative noetherian local ring, $I$ a proper ideal in $R$, and $A$ a non-negative dg algebra over $R$ such that each $A_i$ is a free $R$-module. If $\HH_0(A)\cong R/I$, and $A$ is a minimal dg algebra, meaning $\partial_i(A)\subseteq \fm A_{i-1}$ for all $i$, then 
\begin{equation*}
\level^R A \geq \f_rank_{R/I} (I/I^2)+1.    
\end{equation*}
\end{customprop}
A special case of this proposition occurs when we take $A$ to be the Koszul complex on the ideal $I$. Moreover, if additionally $I/I^2$ is free as an $R/I$-module, then the $R$-level of the Koszul complex on $I$ is exactly $\f_rank_{R/I} (I/I^2)+1$.

The last part of the paper examines examples comparing different lower bounds on levels and investigates whether these bounds can be improved. In particular, it provides a negative answer to the question of whether the method of Avramov, Iyengar, and Neeman in \cite{AIN:2018} can be used to achieve the lower bound \ref{theorem_level_ineq}.

\section{Preliminaries}
Let $R$ be a commutative noetherian ring. By an $R$-complex, we mean a chain complex of $R$-modules; we use lower indexing (homological). We write $\cat{D}(R)$ for the derived category of $R$-modules, which we view as a triangulated category with the usual suspension $\susp$ acting as the translation functor. The homological supremum and infimum of an $R$-complex $M$, are denoted by
\begin{align*}
\sup \HH_*(M)=\sup \set{i\in \mathbb{Z}}{\HH_i(M)\neq 0} \\
\inf \HH_*(M)=\inf \set{i\in \mathbb{Z}}{\HH_i(M)\neq 0}.
\end{align*}
We write $K(\underline{x};M)$ for the Koszul complex on a sequence $\underline{x}=x_1, \ldots, x_n$ over an $R$-complex $M$ and $\HH(\underline{x}; M)$ for its homology. 

\subsection{Local Cohomology} Let $I$ be an ideal and $M$ an $R$-complex. The $I$-power torsion subcomplex of $M$ in degree $i\in \mathbb{Z}$ is defined by
\begin{equation*}
(\Gamma_I M)_i \coloneqq \set{m \in M_i}{ I^n m = 0, \text{ for some } n \geq 0}.
\end{equation*}
The corresponding right derived functor is denoted by $R\Gamma_I (M)$. The local cohomology modules of $M$ supported on $I$ are computed by
\begin{equation*}
\HH^i_I (M) \coloneqq \HH^i(R\Gamma_I (M)) \text{ for } i \in \mathbb{Z}.
\end{equation*}

\subsection{Depth} The \emph{$I$-depth} of $M$ is given by
\begin{equation*}
\depth_R(I,M) \coloneqq \inf \set{i\in \mathbb{Z}}{\HH^i_I (M) \neq 0}
\end{equation*}
and it is infinity if $\HH^i_I (M) = 0$ for all $i$; see \cite{IMSW:2021}. In a local ring $(R, \fm)$, the $\depth_R M$ refers to the $\depth_R(\fm,M)$. From \cite{I:1999}, depth can also be computed using Ext groups and Koszul homology. For Ext groups, we have
\begin{equation*}
\depth_R(I,M) = \inf\set{i\in \mathbb{Z}}{\operatorname{Ext}^i_R  (R/I,M)  \neq 0}.    
\end{equation*} 
If a sequence $\underline{x} \coloneqq x_1, \ldots, x_n$ generates $I$, then we also have
\begin{equation} \label{depth_koszul}
\depth_R(I,M) = n - \sup\set{i\in \mathbb{Z}}{\HH_i(\underline{x};M) \neq 0}.
\end{equation} 

\subsection{Auslander-Buchsbaum equality} \label{Auslander_Buchsbaum} Let $R$ be a commutative noetherian local ring. An $R$-complex $F$ is said to be \emph{perfect} if it is quasi-isomorphic to a finite free $R$-complex. For such an $F$ and any $R$-complex $M$, one has
\begin{equation*}
\depth_R(M\lotimes_R F)=\depth_R M-\projdim_R F.    
\end{equation*}
See \cite[Theorem 2.4]{FI:2003}.

\subsection{Derived complete complexes}  \label{derived_I_complete_tensor_perfect_complex} Let $I$ be an ideal in $R$. The $I$-adic completion of an $R$-complex $M$ is denoted by $\Lambda^I M$ and is defined as the inverse limit
\begin{equation*}
\Lambda^I M\coloneqq  \varprojlim_{n} M/I^n M.
\end{equation*}
The left-derived $I$-completion of $M$ is denoted by $L\Lambda^I (M)$. An $R$-complex $M$ is called \emph{derived $I$-complete} if the natural map
\begin{equation*}
M \overset{}{\longrightarrow} L\Lambda^I (M)    
\end{equation*}
is a quasi-isomorphism; see \cite{CFH:2024,DG:2002, GM:1992, IMSW:2021} for details.

When $F$ is a perfect $R$-complex and $M$ is a derived $I$-complete $R$-complex, then $F\otimes_R M$ is also a derived $I$-complete $R$-complex. This holds due to the following canonical map
\begin{equation*}
N \lotimes_R L\Lambda^I M \longrightarrow L\Lambda^I (N \lotimes_R M),    
\end{equation*}
which becomes isomorphism when $N$ is a perfect $R$-complex; see \cite{GM:1992}.

\begin{proposition}
\cite[Remark 1.7]{IMSW:2021}  \label{depth_sup_ineq} Let $I$ be an ideal and $M$ an $R$-complex. If $\sup \HH_*(M)<\infty $, then the following inequality holds
\begin{equation*}
\depth_R (I, M) \geq -\sup \HH_*(M)\,,
\end{equation*} 
with equality if and only if $\Gamma_I (\HH_{s} (M))\neq 0$, where $s\coloneqq \sup \HH_*( M)$. 
\end{proposition}

The following result derives from the proof of \cite[Theorem 2.7]{IMSW:2021}.

\begin{proposition} \cite[Theorem 2.7]{IMSW:2021}
Let $(R, \fm)$ be a commutative noetherian local ring, $I$ an ideal of $R$ and $M$ a derived $\fm$-complete $R$-complex. Then, the following inequality holds
\begin{equation} \label{depth_ideal_ineq}
\depth_R M \leq \depth_R (I,M) + \dim R/I .
\end{equation} 
In particular, for every prime ideal $\fp$, the following inequality holds
\begin{equation}
\label{depth_prime_ineq}
\depth_R M  \leq \depth_{R_{\fp}} M_{\fp} + \dim R/\fp.
\end{equation}   
\end{proposition}

\subsection{Level} \label{Level} One can define the level of an $R$-complex with respect to any family of $ R$-complexes, but our focus is on the level with respect to $R$. 

\begin{definition} \cite[2.3]{ABIM:2010}
Let $M$ be an $R$-complex. The level of $M$ with respect to $R$, or just $R$-level, is defined inductively as follows
\begin{equation*}
\level^R M \coloneqq \inf
\set{ n\geq 1\ }{ \begin{gathered}
\text{ there is an exact triangle in } \cat{D}(R) \\
K \to L \oplus M \to N\to \susp K \\
\text{with} \level^R K=1 \text{ and } \level^R N= n-1
\end{gathered}}
\end{equation*}
where $\level^R M=0$ if $M$ is quasi-isomorphic to zero, and $\level^R M=1$ if $M$ is built out of $R$ using (de)suspensions, retracts, and finite coproducts.
\end{definition}

An $R$-complex $F$ is perfect if and only if it has finite level with respect to $R$. The length of a perfect complex provides an upper bound for its $R$-level \cite[Lemma 2.5.2]{ABIM:2010}, but it can be arbitrarily larger than the $R$-level. For instance, over any regular local ring, the $R$-level of any perfect $R$-complex cannot exceed the ring's dimension, while there are perfect $R$-complexes with arbitrarily large length; see \cite[Example 5.3]{ABIM:2010}. 

The following result is a special case of \cite[Lemma 2.4 (6)]{ABIM:2010} and provides a comparison of the level after base change.

\begin{proposition} \cite[Lemma 2.4 (6)]{ABIM:2010} \label{level_change_basis}
If $S$ is an $R$-algebra, given the exact functor $-\lotimes_R S \colon \cat{D}(R) \to \cat{D}(S)$, then for any $M\in \cat{D}(R)$, the following inequality holds
\begin{equation*}  
\level^R M\geq \level^S (M \lotimes_R S).
\end{equation*}
\end{proposition}

For an $R$-complex $M$ with nonzero homology, let $P$ be a projective resolution of $M$. For $n \in \mathbb{Z}$, we denote the n-th syzygy of $M$ by $\Omega^R_n(M)\coloneqq \susp^{-n} (P_{\geq n})$. It has been shown in  \cite[Lemma 1.2]{AI:2007} that whether $\HH_0(\Omega^R_n(M))$ is projective is independent of the choice of the projective resolution $P$. The following proposition from \cite{AGMSV:2017} is key to the proof of the level inequality \ref{theorem_level_ineq}.

\begin{proposition} \cite[Theorem 2.1 and Remark 2.5]{AGMSV:2017} \label{level_bound_ineq} Let $M$ be an $R$-complex, with $\HH_i(M) =0$ for all $a <i<b$, $a, b \in \mathbb{Z}$ and $\HH_0(\Omega_{b-1}^R(M))$ is not projective, then
\begin{equation*} 
\level^R M \geq b - a + 1.
\end{equation*}   
\end{proposition} 

In the preceding result, $R$ need not be noetherian. 

\subsection{Balanced big Cohen-Macaulay algebras} The proof of Theorem \ref{theorem_level_ineq} uses the existence of balanced big Cohen–Macaulay algebras. An $R$-algebra $S$ is called \emph{balanced big Cohen–Macaulay algebra} if every system of parameters for $R$ forms an $S$-regular sequence. The existence of such algebras has been proved by Hochster and Huneke in  \cite{H:1975}, \cite{H:2002}, \cite{HH:1991} and \cite{HH:1992} when $R$ is local of equicharacteristic or $\dim R \leq 3$, and in general by André in his recent work on the Direct Summand Conjecture \cite{A:2018}. See also \cite{B:2021} for a different proof by Bhatt when $R$ has mixed characteristic. 

\subsection{Acknowledgments} I would like to thank my advisor, Srikanth B. Iyengar. I am deeply grateful to him for suggesting this problem and for his invaluable guidance, insightful discussions, detailed comments on the manuscript, and constant support throughout the development of this work. I would also like to thank Janina Letz, Des Martin, Claudia Miller, and Josh Pollitz for useful conversations on the material. I am also grateful to the anonymous referee for helpful comments and suggestions. This work is partly supported by National Science Foundation grant DMS-200985. 

\section{The level inequalities}
Here is the theorem from the introduction.

\begin{theorem}
\label{theorem_level_ineq}  Let $R$ be a commutative noetherian local ring, $I$ a proper ideal in $R$ and 
\begin{equation*}
F \colon 0 \longrightarrow F_n \longrightarrow F_{n-1} \longrightarrow \cdots \longrightarrow F_1 \longrightarrow F_0 \longrightarrow 0
\end{equation*}
a finite free $R$-complex with $\HH_0(F) \neq 0$. If $\HH_i(F)$ is $I$-power torsion for $i\geq 1$ and a minimal generator of $\HH_0(F)$ is $I$-power torsion, then the following inequality holds:
\begin{equation*}
\level^R F  \geq \dim R-\dim R/I+1.
\end{equation*}
\end{theorem}

\begin{proof}
From $\HH_0(F)\neq 0$, we have that $\level^R F \geq 1$, so it suffices to prove the inequality for when $\dim R-\dim R/I \geq 1$. Also, by replacing $F$ with its minimal free resolution, we can assume that $F$ is minimal and that $F_n\neq 0$. Next, we take a balanced big Cohen-Macaulay $R$-algebra, and we complete it with respect to $\fm$, obtaining an $\fm$-complete big Cohen-Macaulay $R$-algebra, which we denote by $S$. Set $s\coloneqq \sup \HH_* (F \otimes_R S)$. We claim that the following inequality always holds
\begin{equation}
\label{inequality_with_I}
n\geq \dim R-\dim R/I +s.
\end{equation}
First, consider the case when $s=0$. It follows from $\HH_0(F)$ being finitely generated, Nakayama’s Lemma, and right exactness of the tensor product that each minimal generator of $\HH_0(F)$ gives rise to a nonzero element in $\HH_0(F \otimes_R S)$. By hypothesis, there exists an $I$-power torsion minimal generator of $\HH_0(F)$ and we can lift it to a nonzero $I$-power torsion element of $\HH_0(F \otimes_R S)$, meaning that  $\Gamma_I(\HH_0(F \otimes_R S)) \neq  0$. Therefore, by Proposition \ref{depth_sup_ineq}, we have
\begin{equation*}
\depth_R(I, F \otimes_R S) = -\sup \HH_* (F \otimes_R S) = 0.    
\end{equation*}
The $R$-complex $F \otimes_R S$ is a nonzero derived $\fm$-complete $R$-complex see Subsection \ref{derived_I_complete_tensor_perfect_complex} and applying \eqref{depth_ideal_ineq} yields
\begin{equation*}
\depth_ R ( F \otimes_R S) \leq \dim R/I .    
\end{equation*}
Finally, the Auslander-Buchsbaum formula \ref{Auslander_Buchsbaum}, gives
\begin{equation*}
n\geq \projdim_R F = \depth_R S - \depth_R(F \otimes_R S) \geq \dim R - \dim R/I. 
\end{equation*}
Now, consider the case $s\geq 1$, and take $\fp \in \Ass_R \HH_s(F \otimes_R S)$. Observe that $\HH(F)_{\fp}\neq 0$. We have the following sequence of (in)equalities
\begin{align} 
n &\geq \projdim_{R_\fp} F_\fp \notag\\
&=\depth_{R_\fp} S_\fp -\depth_{R_\fp}( F\otimes_R S)_\fp \notag\\
&= \depth_{R_\fp} S_\fp +s \label{inequality_with_p}\\
& \geq \depth_R S -\dim R/\fp +s \notag \\
&=\dim R- \dim R/\fp  +s. \notag
\end{align}
The first inequality is trivial, while the second follows from \eqref{depth_prime_ineq} applied to $S$, which is a derived $\fm$-complete $R$-complex. The first equality is from the Auslander-Buchsbaum formula \ref{Auslander_Buchsbaum}. The second equality follows from Proposition \ref{depth_sup_ineq}, where we use that $\HH_s(F)$ is $I$-power torsion, hence we can lift a minimal generator of $\HH_s(F)$ to an $I$-power torsion element of $\HH_s(F\otimes_R S)$. The last one comes from $S$ being a big Cohen-Macaulay algebra over $R$. 

The proof of Theorem 2.2 in \cite{CF:2025} shows that $I\subseteq \fp$. Here are the details: we assume towards a contradiction that $I \nsubseteq \fp$. It follows that $F_\fp$ is isomorphic to $\HH_0(F)_\fp$ in the derived category, since $\HH_i(F)$ is $I$-power torsion for $i \geq 1$. This implies that $\sup \HH_*( F_\fp) = 0$, since we additionally have that $\HH(F)_\fp\neq 0$. We then have the following chain of (in)equalities 
\begin{align*}
\depth_{ R_\fp} R_\fp & = \depth_{R_\fp} F_\fp + \projdim_{R_\fp} F_\fp \\
&\geq \projdim_{R_p} F_\fp\\
&\geq \dim R - \dim R/\fp + s\\
&\geq \dim R_\fp + s. 
\end{align*} 
The equality is from the Auslander-Buchsbaum formula \ref{Auslander_Buchsbaum}. The first inequality is trivial, the second follows from \eqref{inequality_with_p}, and the last inequality is standard. Hence, we obtain $\depth R_\fp\geq \dim R_{\fp} +s$, which is a contradiction, since $s$ is positive. 

Therefore, we conclude that $I\subseteq \fp$, and from \eqref{inequality_with_p}, we obtain that
\begin{equation*}
n\geq \dim R -\dim R/I+s.    
\end{equation*}

Next, set $\Omega \coloneqq \HH_0(\Omega_{n-1}^S (F\otimes_R S))$. We claim that $\Omega$ is not projective. Indeed, since $s\leq n-1$, the $S$-complex 
\begin{equation*}
0 \longrightarrow F_n \otimes_R S \longrightarrow F_{n-1} \otimes_R S \longrightarrow 0,    
\end{equation*}
with $F_{n-1} \otimes_R S$ in degree zero, is a free $S$-resolution of $\Omega$. Since $S$ is a big Cohen-Macaulay algebra, we have $\fm S \neq S$, where $\fm$ is the maximal ideal of $R$. Therefore, there exists a maximal ideal $\fn$ of $S$ containing $\fm S$. By $F_n\neq 0$, we have that
\begin{equation*}
\Tor[1] {S} {S/\fn} {\Omega}\cong (S/\fn)\otimes_S (F_n \otimes_R S) \cong (S/\fn)\otimes_R F_n \neq 0.    
\end{equation*}
Thus, $\Omega$ is not flat, and therefore not projective. We now use Proposition \ref{level_bound_ineq} to deduce that
\begin{equation*}
\level^S(F\otimes_R S)\geq n-s+1.
\end{equation*} 
From the base change result in Proposition \ref{level_change_basis}, we then have 
\begin{equation*}
\level^R F \geq \level^S(F\otimes_R S)\geq n-s+1.    
\end{equation*}
Furthermore, from \eqref{inequality_with_I}, one has $n-s\geq \dim R-\dim R/I$, and the proof is complete. 
\end{proof}

Here is an immediate application of Theorem \ref{theorem_level_ineq}.

\begin{proposition}
Let $R$ be a commutative noetherian local ring.
\begin{enumerate}
\item \label{level_koszulcomplex_ideal} If $\underline{x}\coloneqq x_1, \ldots, x_n$ is a generating set for a proper ideal $I$ of $R$, then 
\begin{equation*}
\edim R+1 \geq \level^R K(\underline{x}; R)\geq \dim R -\dim R/I+1.    
\end{equation*}
\item \label{level_koszulcomplex_psop} If $\underline{x}= x_1, \ldots, x_n$ forms a  (partial) system of parameters for $R$, then
\begin{equation*}
\level^R K(\underline{x}; R)= \dim R -\dim R/I+1=n+1.    
\end{equation*}
\end{enumerate}    
\end{proposition}

\begin{proof}
For part \eqref{level_koszulcomplex_ideal}, the inequality on the right comes from applying Theorem \ref{theorem_level_ineq} to the Koszul complex $K(\underline{x}; R)$, which is a perfect $R$-complex with $I=(\underline{x})$-torsion homology. For the inequality on the left, we take a minimal Cohen presentation of $R$, i.e., a surjective map $Q\twoheadrightarrow \hat{R}$, with $Q$ a regular ring and $\edim R=\dim Q$. From \cite[Theorem 5.5]{ABIM:2010} and $Q$ having finite global dimension, we deduce that
\begin{equation*}
\dim Q+1\geq \level^Q K(\underline{x}; Q).    
\end{equation*}
The exact functor $-\lotimes_Q \hat{R}$ along with Proposition \ref{level_change_basis} yield the following inequality
\begin{equation*}
\level^Q K(\underline{x}; Q)\geq \level^{\hat{R}} K(\underline{x}; \hat{R}).    
\end{equation*}
From \cite[Corollary 2.11]{L:2021} and the fact that the completion map $R\to \hat{R}$ is faithfully flat, we deduce that 
\begin{equation*}
\level^{\hat{R}} K(\underline{x}; \hat{R})=\level^{R} K(\underline{x}; R),    
\end{equation*}
which completes this part.

For part \eqref{level_koszulcomplex_psop}, since $\underline{x}$ is a (partial) system of parameters, we have 
\begin{equation*}
\dim R-\dim R/I+1=n+1    
\end{equation*}
and it is enough to show that $\level^R K(\underline{x}; R)\leq \dim R-\dim R/I+1$. Note that the $R$-level of a finite free $R$-complex is always at most its length from \cite[Lemma 2.5.2]{ABIM:2010}. This yields the desired inequality
\begin{equation*}
\level^R K(\underline{x}; R)\leq n+1,    
\end{equation*}
which completes the proof. 
\end{proof}

\begin{remark} Part \eqref{level_koszulcomplex_psop} of the previous proposition also demonstrates that the lower bound on the $R$-level provided by Theorem \ref{theorem_level_ineq} is optimal.  
\end{remark}

\subsection{Free rank} The \emph{free rank} of an $R$-module $M$ is the largest rank of a free direct summand of $M$; it is denoted by $\f_rank_{R}(M)$. The lower bound established in this section is expressed in terms of the free rank of the conormal module $I/I^2$. The conormal module encodes important information about the quotient ring $R/I$, for instance, whether $R/I$ is a complete intersection. Using a different method, we obtain a lower bound on the $R$-level of the Koszul complex, and more generally, on the $R$-level of certain differential graded (dg) algebras; see \cite{A:2010} for definitions and properties of dg algebras. This lower bound relies on Proposition 2.4 of Altmann, Grifo, Montaño, Sanders, and Vue \cite {AGMSV:2017} and on Proposition 2.1 of Iyengar \cite{I:2001}.   

\begin{proposition} \label{level_frank_dgalgebra}
Let $(R,\fm, k)$ be a commutative noetherian local ring, $I$ a proper ideal in $R$, and $A$ a non-negative dg algebra over $R$ such that each $A_i$ is a free $R$-module. If $\HH_0(A)\cong R/I$, and $A$ is a minimal dg algebra, meaning $\partial_i(A)\subseteq \fm A_{i-1}$ for all $i$, then 
\begin{equation*}
\level^R A \geq \f_rank_{R/I} (I/I^2)+1.    
\end{equation*}
In particular, when $I$ is generated by $(\underline{x})$, then 
\begin{equation*}
\level^R K(\underline{x};R) \geq \f_rank_{R/I} (I/I^2)+1.    
\end{equation*}
\end{proposition}

\begin{proof} 
From $\HH_0(A)\neq 0$, we have that $\level^R A \geq 1$, so it suffices to prove the inequality for when $\f_rank_{R/I} (I/I^2) \geq 1$. The natural map $A\to \HH_0(A)\cong R/I$ can be lifted to a morphism of complexes $f\colon A\to G$, where $G$ is a minimal free resolution of $R/I$ over $R$, with differentials denoted by $\partial^G$. Using \cite[Proposition 2.4]{AGMSV:2017}, it is enough to show that $\mathrm{Im} (f_s)\nsubseteq \fm G_s +\mathrm{Ker}(\partial_s^G)$, where $s=\f_rank_{R/I} (I/I^2)$. Since $G$ is a minimal free resolution, we have 
\begin{equation*}
\mathrm{Ker}(\partial_s^G)=\mathrm{Im}(\partial_{s+1}^G)\subseteq \fm G_s.    
\end{equation*}
Hence, $\mathrm{Im} (f_s)\nsubseteq \fm G_s +\mathrm{Ker}(\partial_s^G)$ is equivalent to $\mathrm{Im} (f_s)\nsubseteq \fm G_s$, which is in turn equivalent to $f_s \otimes k\neq 0$. Observe that for $\underline{x}\coloneqq x_1, \ldots, x_n$ a generating set of $I$, there exists a map $K(\underline{x};R) \to R/I$. In degree 0, we get a lift
\begin{equation*}
\begin{tikzcd}
K_0(\underline{x};R) \ar[rr,dashed] \ar[dr] & & A_0\ar[dl]\\
& R/I &
\end{tikzcd}    
\end{equation*}
call it $h_0\colon K_0(\underline{x};R)=R\to A_0$ and we set $a\coloneqq h_0(1)$. For each $i$, the element $x_i a=h_0(x_i)$ lies in $\mathrm{Im}(\partial^A_1)$, hence we have $x_i a =\partial^A_1(a_i)$ for some $a_i \in A_1$. Sending each basis element $e_i \in K_1(\underline{x}; R)=R^n$ to $a_i \in A_1$ determines a morphism of dg $R$-algebras 
\begin{equation*}
h \colon K(\underline{x}; R) \longrightarrow A    
\end{equation*} 
and yields the following commutative diagram by setting $g=f\circ h$
\begin{equation*}
\begin{tikzcd}
K(\underline{x};R) \ar[rr, "h"] \ar[dr, "g" swap] & & A\ar[dl, "f"]\\
& G &
\end{tikzcd}    
\end{equation*}
By tensoring this diagram with the residue field $k$ and taking the $s$-th homology, we obtain the following commutative diagram of modules
\begin{equation*}
\begin{tikzcd}
\HH_s(\underline{x}; k) \ar[rr, "\HH_s(h\otimes k)"] \ar[dr, "\HH_s(g\otimes k)" swap] & & \HH_s(A\lotimes_R k) \ar[dl, "\HH_s(f\otimes k)"]\\
& \Tor [s] R {R/I} k&
\end{tikzcd}    
\end{equation*}
The differential of $K(\underline{x};k)$ is zero, hence $\HH_1(g\otimes k)$ is an isomorphism and $\HH_s(g\otimes k)$ factors as follows
\begin{equation*}
\begin{tikzcd}
\HH_s(\underline{x}; k) \ar[r, "\cong"] \ar[dr, "\HH_s(g\otimes k)" swap] &  \bigwedge^s \Tor[1] R {R/I} k \ar[d, "\kappa_s"]\\
& \Tor [s] R {R/I} k
\end{tikzcd}    
\end{equation*}
where $\kappa \colon \bigwedge \Tor[1] R {R/I} k\to\Tor R {R/I} k$ is the natural map of graded $k$-algebras. We claim that $\kappa_s\neq 0$, which implies that $\HH_s(f\otimes k)\neq 0$, and hence that $f_s \otimes k \neq 0$. 

To prove the claim, we first apply \cite[Proposition 2.1]{I:2001} to $G'\otimes_R k$, where $G'$ is a dg $\Gamma$-algebra resolution of $R/I$ and we get the following isomorphism of $k$-algebras
\begin{equation*}
G' \otimes_R k \cong B \otimes_{k} \Lambda,    
\end{equation*}
where $B$ is a graded $k$-algebra with $B_0=k$ and
\begin{equation*}
\Lambda =\bigwedge \bigg( \bigoplus_{i=1}^s k y_i \bigg),    
\end{equation*}
where $s=\f_rank_{R/I} (I/I^2)$, and for each $i=1, \ldots, s$ the element $y_i$ has degree $1$ and satisfies $\partial^\Lambda (y_i)=0$. Taking homology of $B\otimes_k \Lambda$, we observe that
\begin{equation*}
\Tor[1] R {R/I} k =(B_1 \otimes_k k) \oplus  \bigg( k\otimes_k \bigoplus_{i=1}^s k y_i \bigg)=B_1\oplus \bigg(\bigoplus_{i=1}^s k y_i \bigg).    
\end{equation*} 
Also, $\Lambda_s$ is a direct summand of $\Tor[s] R {R/I} k$. Since the element $y_1\wedge \ldots\wedge y_s \in \Lambda_s$ is nonzero, it follows that $\kappa_s\neq 0$ and the proof has been completed.
\end{proof}

\subsection{Lech-Independent sequences} A sequence $\underline{x}=x_1, \ldots, x_n$ is called \emph{Lech-independent} if for $I=(\underline{x})=(x_1,\ldots, x_n)$, the natural surjection
\begin{equation*}
(R/I)^n \longrightarrow I/I^2    
\end{equation*}
is an isomorphism, see \cite{B:2023, Ha:2001, L:1964}. A regular sequence is also Lech-independent, and the converse holds when, additionally, $I$ has finite projective dimension, see \cite{V:1967}. Another example of a Lech-independent sequence is a minimal generating set of the maximal ideal of a local ring.  Lech-independent sequences are closed under flat base change.

The following result is immediate from Proposition \ref{level_frank_dgalgebra} and generalizes the \cite[Theorem 4.2(1)]{AGMSV:2017}, which concerns the case of the maximal ideal. 

\begin{corollary} \label{level_koszulcomplex_indep} Let $(R,\fm, k)$ be a commutative noetherian local ring and $\underline{x}=x_1,\ldots, x_n$ a Lech-independent sequence, then 
\begin{equation*}
\level^R K(\underline{x};R) =n+1.    
\end{equation*}
\end{corollary}

\begin{proof}
It is clear that $\level^R K(\underline{x}; R)\leq n+1$ from \cite[Lemma 2.5.2]{ABIM:2010}. For the reverse inequality, we apply Proposition \ref{level_frank_dgalgebra} to the sequence $\underline{x}=x_1, \ldots, x_n$. From $\underline{x}$ being a Lech-independent sequence, $\f_rank_{R/I}(I/I^2)=n$, and this completes the proof. 
\end{proof}

\section{Examples}
This section examines examples comparing different lower bounds on levels and investigates whether these bounds can be improved. 

With stronger assumptions on the complex, a sharper lower bound than the one in \cite[Theorem 4.2]{AIN:2018}—involving the superheight of an ideal—was established in \cite[Theorem 5.1]{ABIM:2010}. The \emph{superheight} of an  ideal $I$ is defined as the number 
\begin{equation*}
\superheight I = \sup \set{\height IS}{ S \text{ is a noetherian } R\text{-algebra}}.   
\end{equation*}
The following result is a special case of \cite[Theorem 5.1]{ABIM:2010} and has also been proved in \cite[Theorem 3.2]{AGMSV:2017}. Using Theorem \ref{theorem_level_ineq}, the proof is significantly simplified. 

\begin{corollary}
Let $R$ be a commutative noetherian local ring,  $F$ a perfect $R$-complex and $I\subseteq R$ the annihilator of $\bigoplus_{i\in \mathbb{Z}} \HH_i(F)$. Then, the following inequality holds
\begin{equation*}
\level^R F \geq \superheight I + 1.
\end{equation*}
\end{corollary}

\begin{proof}
Let $S$ be a noetherian $R$-algebra. The ideal $IS$ annihilates 
\begin{equation*}
\bigoplus_{i \in \mathbb{Z}} \HH_i(S \otimes_R F),    
\end{equation*}
hence, by applying the previous theorem, we get the first inequality
\begin{equation*}
\level^S(S\otimes_R F) \geq \dim S -\dim S/IS +1\geq \height IS+1.
\end{equation*}
The second inequality is standard. Additionally, from Proposition \ref{level_change_basis}, we obtain that 
\begin{equation*}
\level^R F\geq  \level^S(S\otimes_R F),   
\end{equation*}
which completes the proof.
\end{proof}

Another interesting number is the \emph{bigheight} of an ideal $I$ defined as follows:
\begin{equation*}
\bigheight I = \sup \set{\height \fp }{\fp \text{ is minimal prime over } I}.    
\end{equation*}
It is easy to check that we always have the following inequalities
\begin{equation*}
\height I\leq \bigheight I\leq \superheight I
\end{equation*}
and that the first one can be strict. Additionally, it was shown in \cite{K:1988} that the second inequality can also be strict. We recall that the following inequality always holds 
\begin{equation*}
\height I\leq \dim R-\dim R/I.    
\end{equation*}
However, we observe that the quantities $\dim R-\dim R/I$ and $\bigheight I$ (or even $\superheight I$) cannot be compared. 

On the one hand, take the ring $R=k\llbracket x_1, \ldots, x_n\rrbracket$ and the ideal $I=(x_1)\cap(x_2,\ldots, x_n)$ for which we have 
\begin{equation*}
1=\dim R-\dim R/I \quad \text{and}\quad \bigheight I=n-1.    
\end{equation*}
On the other hand, the ring $R=k\llbracket x_1, \ldots, x_n \rrbracket/ ((x_1)\cap (x_2,\ldots,x_n))$ and the ideal $I=(x_2, \ldots, x_n)$ give
\begin{equation*}
n-2=\dim R-\dim R/I \quad \text{and}\quad \bigheight I=0.    
\end{equation*}
This last example also shows that the quantity $\dim R-\dim R/I$ can be arbitrarily larger than $\height I$. 

\begin{remark}
One could ask whether, under the conditions of Theorem \ref{theorem_level_ineq}, the quantity $\dim R - \dim R/I$ could be replaced by $\superheight I$ or $\bigheight I$. The answer is negative, as we can see by considering the ring $R=k\llbracket x_1, \ldots, x_n \rrbracket$ for $n\geq 3$, the ideal $I=(x_1)\cap (x_2,\ldots, x_n)$ and the perfect $R$-complex $F=K(x_1; R)$. It is clear that $F$ has $I$-torsion homology and $\level^R F=2$, but $\bigheight I=n-1$.
\end{remark}

\subsection{Tensor nilpotent and fiberwise zero maps} A morphism $f\colon X \to Y$ in $\cat{D}(R)$ is called \emph{tensor nilpotent} if for some $n \in \mathbb{N}$ the map 
\begin{equation*}
\otimes^n f\colon X \lotimes_R \cdots \lotimes_R  X \longrightarrow  Y \lotimes_R \cdots \lotimes_R Y
\end{equation*} 
is zero in $\cat{D}(R)$. A morphism $f\colon X \to Y$ in $\cat{D}(R)$ is called \emph{fiberwise zero} if for all $\fp \in \Spec(R)$ the map $k(\fp)\lotimes_R f$ is zero in $\cat{D}(k(\fp))$; see \cite[3.2]{AIN:2018}. These two notions are equivalent when the map $f$ is between perfect $R$-complexes, see \cite{H:1987,N:1992}. In this subsection, we will investigate the relation between Theorem \ref{theorem_level_ineq} and \cite[Theorem 4.1]{AIN:2018}. This result states that over a commutative noetherian ring $R$, if there exists a morphism $f \colon G \to F$ of perfect $R$-complexes, which is not fiberwise zero and factors through an $R$-complex with $I$-torsion homology for some ideal $I \subseteq R$, then the following inequality holds
\begin{equation}
\label{theorem_AIN_4.1}
\level^R \Hom R G F \geq  \height I + 1.    
\end{equation}
This result led to the version \cite[Theorem 4.2]{AIN:2018} of the Improved New Intersection Theorem. We observe that under these conditions, we cannot replace $\height I$ with $\dim R-\dim R/I$ in \ref{theorem_AIN_4.1}. To see this, we take the ring 
\begin{equation*}
R=k \llbracket x_1, \ldots, x_n \rrbracket/((x_1)\cap (x_2,\ldots,x_n))\,,    
\end{equation*}
the perfect $R$-complexes $F=G=R$ and the ideal $I=(x_2, \ldots, x_n)$. Consider the map $f\colon R\xrightarrow{x_1} R$, which is not tensor nilpotent since $x_1$ is not a nilpotent element in $R$. From the following diagram, one can observe that the map $f$ can be factored through the $I$-torsion $R$-complex $K(x_2, \ldots, x_n; R)$.
\begin{equation*}
\begin{tikzcd}
& & & &  0 \ar[r] \ar[d] & R \ar[r] \ar[d,"\mathrm{id}"] & 0 \\
0 \ar[r] & R \ar[r] & R^{n-1} \ar[r]  & \cdots \ar[r] & R^{n-1} \ar[r]  \ar[d] & R \ar[r] \ar[d,"x_1"] & 0 \\
& & & & 0 \ar[r] & R \ar[r]  & 0
\end{tikzcd}
\end{equation*}    
We see that  $\level^R \Hom R R R=\level^R R=1$, while the quantity 
\begin{equation*}
\dim R-\dim R/I +1= n-1    
\end{equation*}
can be arbitrarily large. 

Another natural question is whether, in \ref{theorem_AIN_4.1}, one could replace $\height I$ with $\bigheight I$. The answer is again negative, as we can see by considering the ring $R=k \llbracket x_1, \ldots, x_n \rrbracket$, the ideal $I=(x_1)\cap (x_2,\ldots, x_n)$, and the perfect $R$-complexes $F=G=K(x_1; R)$. Then, we observe that the identity map on $F$ factors through the $I$-power torsion $R$-complex, $R\Gamma_I F$, and is not a fiberwise zero map. In this case, we have 
\begin{equation*}
\level^R \Hom R F F =\level^R F=2,    
\end{equation*}
but $\bigheight I +1 =n$. 

\bibliographystyle{amsplain}

\end{document}